\documentclass[12pt]{article}

\usepackage{amsmath,amsthm,amsfonts,amssymb,latexsym,amscd}

\input{prepictex.tex}
\input{pictex.tex}
\input{postpictex.tex}

\begin{document}

\newtheorem*{theo}{Theorem}
\newtheorem*{pro}{Proposition}
\newtheorem*{cor}{Corollary}
\newtheorem*{lem}{Lemma}
\newtheorem{theorem}{Theorem}[section]
\newtheorem{corollary}[theorem]{Corollary}
\newtheorem{lemma}[theorem]{Lemma}
\newtheorem{proposition}[theorem]{Proposition}
\newtheorem{conjecture}[theorem]{Conjecture}
\newtheorem{definition}[theorem]{Definition}
\newtheorem{problem}[theorem]{Problem}
\newtheorem{remark}[theorem]{Remark}
\newtheorem{example}[theorem]{Example}
\newcommand{\Naturali}{{\mathbb{N}}}
\newcommand{\Reali}{{\mathbb{R}}}
\newcommand{\Complessi}{{\mathbb{C}}}
\newcommand{\Toro}{{\mathbb{T}}}
\newcommand{\Relativi}{{\mathbb{Z}}}
\newcommand{\HH}{\mathfrak H}
\newcommand{\KK}{\mathfrak K}
\newcommand{\LL}{\mathfrak L}
\newcommand{\as}{\ast_{\sigma}}
\newcommand{\tn}{\vert\hspace{-.3mm}\vert\hspace{-.3mm}\vert}
\def\mA{{\mathfrak A}}
\def\A{{\mathcal A}}
\def\mB{{\mathfrak B}}
\def\B{{\mathcal B}}
\def\C{{\mathcal C}}
\def\D{{\mathcal D}}
\def\F{{\mathcal F}}
\def\H{{\mathcal H}}
\def\J{{\mathcal J}}
\def\K{{\mathcal K}}
\def\L{{\cal L}}
\def\N{{\cal N}}
\def\M{{\cal M}}
\def\O{{\mathcal O}}
\def\P{{\cal P}}
\def\S{{\cal S}}
\def\T{{\cal T}}
\def\U{{\cal U}}
\def\W{{\cal W}}
\def\b{\lambda_B(P}
\def\j{\lambda_J(P}
\def\uu{\underline{u}}

\title{Automorphisms of the UHF algebra that do not extend to the Cuntz algebra}

\author{Roberto Conti}

\date{}
\maketitle

\renewcommand{\sectionmark}[1]{}

\noindent
{\small \date{\today}}

\vspace{7mm}
\begin{abstract}
Automorphisms of the canonical core UHF-subalgebra $\F_n$
of the Cuntz algebra $\O_n$
do not necessarily extend to automorphisms of $\O_n$.
Simple examples are discussed within the family of infinite tensor products of 
(inner) automorphisms of the matrix algebras $M_n$.
In that case, 
necessary and sufficient conditions for the extension property are presented.
It is also addressed the problem of extending to $\O_n$ the automorphisms
of the diagonal $\D_n$, which is a regular MASA with Cantor spectrum.
In particular, it is shown the existence of product-type automorphisms of $\D_n$
that are not extensible to (possibly proper) endomorphisms of $\O_n$.
\end{abstract}

\vfill

\vspace{3mm}

\newpage

\section{Introduction}
Among $C^*$-algebras, the class of UHF algebras is perhaps one of the first nontrivial 
(i.e., infinite-dimensional) examples that comes to mind and it has been investigated in 
depth in the outstanding work by Glimm in the early 60s, \cite{Gl}. 
Since then, the study of $C^*$-algebras has enjoyed a dramatic increase in perspectives and results.
At some point later the Cuntz algebras $\O_n$ made their appearence on stage, \cite{Cun1}. 
These $C^*$-algebras are quite different from UHF algebras as, for instance, they are traceless (and purely infinite).
However, it is well-known that $\O_n$ contains
in canonical fashion a copy of the UHF algebra of type $n^\infty$
(that is, the $C^*$-algebraic tensor product of countably infinite 
copies of $M_n$, the $n \times n$ matrix algebra over $\mathbb C$)
as the canonical core UHF-subalgebra $\F_n$.

In contrast to the case of $\F_n$, the study of automorphisms of $\O_n$ is quite challenging.
Typical problems range from the construction of explicit examples to global properties of the group 
${\rm Aut}(\O_n)$ and some notable subgroups and quotients.

Among other things, the group ${\rm Aut}(\O_n,\F_n)$ of all automorphisms of $\O_n$
leaving $\F_n$ globally invariant has sporadically made its appearence in the literature, \cite{Cun2,CS,CRS}. 
(Besides, it turns out that 
${\rm Aut}(\O_n,\F_n) = {\rm Aut}(\O_n) \cap \{\lambda_u \ | \ u \in \U(\F_n)\}$
and thus the examples of permutation automorphisms of $\O_n$ exhibited in \cite{CS} 
restrict to automorphisms of $\F_n$ that are also easily seen to be outer.)
In this short note, we will add another piece of information 
by showing that the canonical restriction map 
\begin{equation} 
r: {\rm Aut}(\O_n,\F_n) \to {\rm Aut}(\F_n)
\end{equation}
is not surjective. (By \cite[Corollary 4.10]{CRS}, the kernel of this group 
homomorphism consists precisely of the gauge automorphism.) A 
few remarks are in order. First of all, it is plain that all inner automorphisms of $\F_n$ extend to
(inner) automorphisms of $\O_n$. Secondly, it is well known (see e.g. \cite[Coro IV.5.8]{Dav}) that
${\rm Aut}(\F_n) = \overline{\rm Inn}(\F_n)$, that is every automorphism of $\F_n$ is approximately inner
in the sense that it is the limit of inner automorphisms of $\F_n$ in the topology of pointwise norm-convergence.
Therefore our result in particular shows the existence of sequences 
of inner automorphisms of $\O_n$ that converge 
pointwise in norm 
on $\F_n$ but not on $\O_n$. Moreover it makes it clear that there are strict inclusions 
(cf. \cite{Arc})
\begin{equation}
{\rm Inn}(\F_n) \subsetneq r\big({\rm Aut}(\O_n,\F_n)\big) \subsetneq {\rm Aut}(\F_n).
\end{equation}
Finally, we deduce that there are automorphisms of $\F_n$ that do not extend to $\O_n$ even
as proper endomorphisms as the latter possibility has been ruled out in \cite[Corollary 4.9]{CRS}. 
On the contrary, this last fact is not true for the canonical Cartan subalgebra $\D_n$ of $\O_n$
for there are proper endomorphisms of $\O_n$ that restrict to automorphisms of $\D_n$, see \cite{CS}.
Still one can show the existence of automorphisms of $\D_n$ that do not extend to endomorphisms of $\O_n$.

\medskip

A few words on the notation (cf. \cite{CS}): 
for each integer $n \geq 2$
the Cuntz algebra $\O_n$ is the universal $C^*$-algebra generated by $n$ isometries 
$S_1, \ldots, S_n$ whose ranges sum up to 1. 
There is a one-to-one correspondence $v \mapsto \lambda_v$ between
unitaries in $\O_n$ and unital $*$-endomorphism of $\O_n$ 
that associates to $v \in \U(\O_n)$ the endomorphism
$\lambda_v$ defined by $\lambda_v(S_i)=vS_i$, $i=1,\ldots,n$. 
The canonical core UHF-subalgebra $\F_n$ is the norm-closure of the union 
of the matrix subalgebras 
$\F_n^k \simeq M_n \otimes \ldots \otimes M_n$ ($k$ factors), where 
$$\F_n^k 
= {\rm span}\{{S_{\alpha_1}\ldots S_{\alpha_k}{S_{\beta_k}^* \ldots 
S_{\beta_1}^*, \ 1 \leq \alpha_1,\ldots,\alpha_k,\beta_1,\ldots,\beta_k \leq n}}\} \ . $$
$\F_n$ possesses a unique normalized trace, which we denote $\tau$. 
$\D_n$ is the diagonal subalgebra of $\F_n$
and it is the norm-closure
of the union of the commutative algebras 
$\D_n^k = \D_n \cap \F_n^k \simeq {\mathbb C}^{n^k}$. We denote by 
 $\varphi$ the canonical endomorphism of $\O_n$, defined by
$\varphi(x)=\sum_{i=1}^n S_i x S_i^*$. This shift endomorphism satisfies 
$S_i a = \varphi(a) S_i$ for all $a \in \O_n$ and all $i=1,\ldots,n$.
For a unital $C^*$-algebra $B$, ${\rm Aut}(B)$ and ${\rm Inn}(B)$
are the groups of automorphisms and inner automorphisms of $B$, respectively.
For unital $C^*$-algebras $A \subset B$,
${\rm Aut}(B,A)$ and ${\rm Aut}_A(B)$ are the group of automorphisms
of $B$ leaving $A$ globally and pointwise invariant, respectively.


\section{Main result}
We refer to \cite[Chapter VI]{EvKa} for generalities about the UHF 
algebras and their automorphisms.
In order to simplify the notation we will often represent elements of $\F_n$ 
by tensor products of matrices,
through the canonical identification of $\F_n$ with $\otimes_{i=1}^\infty M_n$.

Given a sequence of unitaries $\uu = (u_i)$ with $u_i \in \U(\F_n^1) \simeq U(n)$,
we consider the associated automorphism $\alpha_{\uu}$ of $\F_n$ such that, 
for all multiindices $\alpha=(\alpha_1,\ldots,\alpha_k)$ and 
$\beta=(\beta_1,\ldots,\beta_k)$
of the same length $k$ and for all $k \geq 1$,
\begin{equation}
\alpha_{\uu} (S_\alpha {S_\beta}^*) 
= u_1 S_{\alpha_1} u_2 S_{\alpha_2} \cdots u_k S_{\alpha_k}
S_{\beta_k}^* u_k^* \cdots S_{\beta_2}^*u_2^* S_{\beta_1}^* u_1^*
\end{equation}
In the tensor product picture $\alpha_{\uu}$ is nothing but the infinite tensor product
automorphism $\otimes_{i=1}^\infty {\rm Ad}(u_i)$. It is also clear that if 
$\lim_{k\to \infty} u_1 \otimes u_2 \otimes \ldots \otimes u_k \otimes 1 \otimes 1 \otimes \ldots =: u$
exists in $\otimes_i M_n$ then $\otimes_i {\rm Ad}(u_i) = {\rm Ad}(u)$ is inner, 
while the converse holds true whenever $1 \in \sigma(u_i)$, for all $i$'s
(the latter assumption can always be satisfied by rotating the $u_i$'s if necessary),
see \cite[Theorem 6.3]{EvKa}.

Of course, if $\alpha_{\uu} = {\rm Ad}(u)$ is inner then it extends to an inner automorphism
of $\O_n$, namely $\lambda_{u \varphi(u^*)}$. It is quite possible that even though 
$\alpha_{\uu}$ is outer, it still extends to an automorphism of $\O_n$. A simple such 
example arises in the case where $\uu$ is a nonscalar constant sequence: 
$u_1 = u_2 = u_3 = \ldots$. Then $\alpha_{\uu} = \otimes_i {\rm Ad}(u_1)$ is 
outer on $\F_n$ and extends to the (still outer) 
Bogolubov automorphism $\lambda_{u_1}$ of $\O_n$.
At this point one could suspect that the possibility of  
extending $\alpha_{\uu}$ to $\O_n$ depends on whether 
\begin{equation}\label{eqn}
\lim_{k \to \infty} u_1 \otimes u_2 u_1^* \otimes u_3 u_2^* \otimes \ldots \otimes u_{k+1}u_k^* 
\otimes 1 \otimes 1 \ldots
\end{equation}
exists in $\otimes_i M_n$. This is indeed the case, as the following theorem shows. 

\begin{theorem}\label{mainthm}
Let $u_1, u_2, u_2, \ldots$ be an infinite sequence of unitaries in $\U(\F_n^1)$. 
If the limit in equation (\ref{eqn}) exists and thus defines a unitary $v \in \F_n$ then $\lambda_v$
is an automorphism of $\O_n$ such that, in restriction to $\F_n$, $\lambda_v$ 
coincides with $\alpha_{\uu} = \otimes _i {\rm Ad}(u_i)$.
Conversely, suppose that $\alpha_{\uu}$ extends to an endomorphism of $\O_n$.
Then there are phases $(e^{i \theta_k})_k$ such that
the limit in eq. (\ref{eqn}) exists for the sequence $(u'_k)$, 
where $u'_k = e^{i\theta_k} u_k$.
\end{theorem}

\begin{proof}
If $v$ is defined by the limit as above, then one can 
easily check that $\lambda_v$ coincides with $\alpha_{\uu}$ in restriction to $\F_n$,
and thus $\lambda_v \in {\rm Aut}(\O_n)$ by \cite[Corollary 4.9]{CRS}. 

Conversely, let $\lambda_w$ be an endomorphism of $\O_n$ such that $\lambda_w(x)=\alpha_{\uu}(x)$
for all $x \in \F_n$. Then $\lambda_w \in {\rm Aut}(\O_n)$ and $w \in \F_n$. Now, 
$$w S_i S_j^* w^* = u_1 S_i S_j^* u_1^*$$ for all $i,j=1,\ldots,n$ and thus 
$w^* u_1 \in (\F_n^1)' \cap \F_n = \varphi(\F_n)$, 
that is $w=u_1 \varphi(z_1)$ for some $z_1 \in \U(\F_n)$. 
But then 
$$w \varphi(w) S_i S_j S_k^* S_h^* \varphi(w^*)w^* 
= u_1 \varphi(u_2)S_i S_j S_k^* S_h^* \varphi(u_2^*)u_1^*$$
for all $i,j,h,k =1,\ldots,n$, 
that is $\varphi(w^*)w^*u_1\varphi(u_2) \in (\F_n^2)' \cap \F_n = \varphi^2(\F_n)$.
By the above, this means that $w^* z_1^* u_2 = \varphi(z_1^*)u_1^*z_1^*u_2 \in \varphi(\F_n)$. 
Thus $u_1^* z_1^* u_2 = \varphi(z_2^*) \in \varphi(\F_n)$ for some unitary $z_2 \in \U(\F_n)$
from which $z_1^* = u_1 \varphi(z_2^*)u_2^* = \varphi(z_2^*)u_1 u_2^*$ and consequently
$$w = u_1 \varphi(u_2u_1^*)\varphi^2(z_2) \ .$$
Repeating this argument one obtains that, for each positive integer $k$,
$$w = u_1 \varphi(u_2 u_1^*)\varphi^2(u_3 u_2^*) 
\cdots \varphi^k(u_{k+1}u_k^*)\varphi^{k+1}(z_{k+1})$$
for a suitable unitary $z_k \in \F_n$.
Moreover, one can find a sequence $(e^{i \theta_k})_k$ such that, 
after replacing $u_k$ with $e^{i\theta_k}u_k$, 
one can always assume that $\tau(z_k) \in {\mathbb R}_+$.
Consider now the $\tau$-invariant conditional expectation $E_k: \F_n \to \F_n^k$, 
$E_k = {\rm id}_k \otimes \tau$, with ${\rm id}_k$ the identity on $\F_n^k$,
for which one has $x = \lim_k(E_k(x))$ for all $x \in \F_n$.
By the above, it is clear that, for all $k \geq 1$, 
$E_k(w) = u_1 \otimes \ldots \otimes u_{k+1}u_k^* \tau(z_{k+1})$
and thus
$1 = \|w\| = \lim_k \|u_1 \otimes \ldots \otimes u_{k+1}u_k^*\| \tau(z_{k+1}) 
= \lim_k \tau(z_{k+1})$.
That is,
$$\lim_k \|w - u_1 \otimes \ldots \otimes u_{k+1}u_k^*\| 
= \lim_k \|w - u_1 \otimes \ldots \otimes u_{k+1}u_k^* \tau(z_{k+1})\| 
= 0 
\ . $$
\end{proof}

Automorphisms of $\F_n$ of the form $\alpha_{\uu}$ satisfy
$\alpha_{\uu}(\F_n^k) = \F_n^k$ for every $k$.
An endomorphism of the Cuntz algebra is said to be localized 
if it is induced by a unitary in some finite matrix algebra, 
i.e. in $\bigcup_k \F_n^k$, \cite{CP}.

\begin{corollary}
Suppose that $\alpha_{\uu}$ extends to an automorphism of the 
Cuntz algebra. Then the following conditions are equivalent:
\begin{itemize}
\item[(a)] one extension of $\alpha_{\uu}$ is localized;
\item[(b)] all extensions of $\alpha_{\uu}$ are localized;
\item[(c)] one has eventually $u_{k+1}u_k^* \in {\mathbb T}1$.
\end{itemize}
\end{corollary}

\begin{proof}
This follows from the fact that any two such extensions must differ by a 
gauge automorphism and also that a unitary implementing one of them is obtained
through equation (\ref{eqn}) up to phases, as explained above.
\end{proof}

\section{Outlook}

We would like to mention a few related problems. 
On one hand, one should investigate the detailed structure of ${\rm Aut}(\O_n,\F_n)$ and 
find an intrinsic characterization of automorphisms of $\F_n$ that extend to $\O_n$.
On the other hand, one should study the analogous problems for the extension of automorphisms 
from the diagonal $\D_n$ to $\F_n$ and from $\D_n$ to $\O_n$.

In this respect, we can add something more about extension of automorphisms
of $\D_n$ to $\O_n$. Especially, we are now ready to show that 
not all automorphisms of $\D_n$ are extensible to (possibly proper) endomorphisms of $\O_n$. 
This will easily follow also from Theorem \ref{endaut} below, but the following 
observation is more in line with the criterion given in Theorem \ref{mainthm}. 

\begin{proposition}\label{diagonal}
Let $\alpha$ be a product type automorphism of the diagonal, 
i.e. $\alpha(\varphi^k(\D_n^1))=\varphi^k(\D_n^1)$ for all $k \geq 0$. Then $\alpha$ 
extends to a (possibly proper) endomorphism of $\O_n$ if and only if the action on each 
$\varphi^k(\D_n^1)$ is eventually identical. In that case, $\alpha$ extends to a permutation 
automorphism of $\O_n$. 
\end{proposition}
\begin{proof}
Suppose that $\alpha$ extends to an endomorphism $\lambda_u$
of $\O_n$. Then we have $\lambda_u(\varphi^k(\D_n^1))=\varphi^k(\D_n^1)$ for all $k$ and an
easy induction shows that ${\rm Ad}(u)$ preserves each $\varphi^k(\D_n^1)$. Thus, in
particular, $u$ is in the normalizer of $\D_n$. Consequently (see \cite{Pow}), $u=wv$ with
$w\in \S_n$ and $v\in \U(\D_n)$, where $\S_n$ is the subgroup of $\U(\O_n)$
of unitaries that can be written as finite sum of words in $S_i$ and $S_i^*$. 
But then it follows that the restriction $\alpha$ of $\lambda_u$ to $\D_n$ coincides 
with the restriction of $\lambda_w$ to $\D_n$. Thus we have that ${\rm Ad}(w)$ 
preserves each of $\varphi^k(\D_n^1)$. It follows that the restriction to $\D_n$ 
of the trace $\tau$ is ${\rm Ad}(w)$-invariant. 
This however is only possible if $w$ belongs to $\F_n$, that is $w$ is a
permutation matrix. Since ${\rm Ad}(w)$ preserves each of $\varphi^k(\D_n^1)$ it follows 
that we have $w=w_1 \varphi(w_2) \varphi^2(w_3) \cdots \varphi^r(w_{r+1})$ for some 
positive integer $r$ and permutation matrices $w_j$ in $\F_n^1$, $j=1,\ldots,r$. 
This implies that the restriction of $\lambda_w$ to $\varphi^k(\F_n^1)$ 
coincides with such restriction of ${\rm Ad}\big(\varphi^k(w_{k+1}w_k \cdots w_2 w_1)\big)$. 
This means that $\lambda_w$ in restriction to $\F_n$ is 
a product type automorphism of $\F_n$ for which the limit in equation (\ref{eqn})
equals 1 (actually the terms of the sequence eventually stabilize), and
thus by Theorem \ref{mainthm} $\lambda_w$ is (and thus $\alpha$ extends to) an
automorphism of $\O_n$.
\end{proof}

We denote by ${\rm EndAut}(\O_n,\D_n)$ the subsemigroup of ${\rm End}(\O_n)$
of those endomorphisms that restrict to automorphisms of $\D_n$
and by ${\rm End}_{\D_n}(\O_n)$ the subsemigroup of ${\rm EndAut}(\O_n,\D_n)$ 
of those endomorphisms acting trivially on $\D_n$.
By the analysis in \cite{CS}, it is already known that 
$${\rm Aut}(\O_n,\D_n) \subsetneq {\rm EndAut}(\O_n,\D_n) \ . $$
However, the following result is already implicit in \cite{Cun2}.

\begin{proposition} \label{fixed}
With the above notation one has
$${\rm Aut}_{\D_n}(\O_n) = {\rm End}_{\D_n}(\O_n).$$
\end{proposition}

\begin{proof} 
Let $u$ be a unitary in $\O_n$. If $\lambda_u(x) = x$ for all $x \in \D_n$ then
it is easy to see by induction on $k$ that $u$ commutes with $\varphi^k(\D_n^1)$
for all $k \geq 1$, 
and therefore
$$u \in \Big(\bigcup_{k \geq 0} \varphi^k(\D_n^1)\Big)' \cap \O_n = \D_n ' \cap \O_n \ . $$
$\D_n$ being a MASA in $\O_n$ one thus has that $u \in \D_n$ 
and therefore $\lambda_u \in {\rm Aut}_{\D_n}(\O_n)$ by \cite{Cun2}.
\end{proof}

\begin{theorem}\label{endaut} 
The restriction map
\begin{equation}
r: {\rm EndAut}(\O_n,\D_n) \to {\rm Aut}(\D_n)
\end{equation}
is not surjective. 
Furthermore, $r({\rm EndAut}(\O_n,\D_n))$ is not a subgroup of ${\rm Aut}(\D_n)$,
and it is the disjoint union of the subgroup of those automorphisms extensible to automorphisms of $\O_n$,
and the subsemigroup of those automorphisms extensible to proper endomorphisms of $\O_n$.
\end{theorem}

\begin{proof}
The first claim follows from Proposition \ref{diagonal}. 
Let $u$ be a unitary in $\O_n$ such that $\lambda_u$ is a proper endomorphism of $\O_n$ and 
$\lambda_u(\D_n)=\D_n$ (cf. \cite{CS}). We claim that $(\lambda_u |_{\D_n})^{-1} \in {\rm Aut}(\D_n)$ 
is not extensible to  an endomorphism of $\O_n$. For let $\lambda_v$ be such an extension. 
Then $\lambda_v \lambda_u(x) = x$ for all $x \in \D_n$ and thus, by Proposition \ref{fixed}, 
$\lambda_v \lambda_u$ is an automorphism of $\O_n$, a contradiction.
A similar argument also shows that if two endomorphism $\lambda_{u_1}$ and $\lambda_{u_2}$ of $\O_n$
restrict to the same automorphism of $\D_n$ then they are either both automorphisms or both 
proper endomorphisms.
\end{proof}
By the above, automorphisms of $\D_n$ obtained by restriction of proper endomorphisms of $\O_n$ 
are necessarily not of product type.

\medskip
 
Finally, we note that the subgroup $r({\rm Aut}(\O_n,\D_n))$ is not 
normal in ${\rm Aut}(\D_n)$. 
Indeed 
by Gelfand duality one has
${\rm Aut}(\D_n)\simeq {\rm Homeo}({\mathcal C})$, 
where $\mathcal C$ is the Cantor set,
and the latter group is known to be simple.

\vspace{5mm}

\noindent
Roberto Conti\\
Dipartimento di Scienze \\
Universit\`a di Chieti-Pescara ``G. D'Annunzio''\\
Viale Pindaro 42, I-65127 Pescara, Italy \\
E-mail: conti@sci.unich.it \\

\end{document}